\documentclass{monsky2009}

\usepackage{amssymb}

\usepackage{bm}

\newenvironment{conjecture*}[1][]{\textbf{Conjecture #1\hspace{.3em}}}{}
\newenvironment{theorem*}[1]{\textbf{#1}\itshape \hspace{.3em}}{\upshape}
\newenvironment{remark*}[1]{\textbf{#1}\itshape \hspace{.3em}}{\upshape}
\newenvironment{example*}[1]{\textbf{#1}\itshape \hspace{.3em}}{\upshape}
\newenvironment{proof}[1][]{\textbf{Proof #1\hspace{.3em}}}{}

\newenvironment{proofsketch}{\textbf{Proof Sketch\hspace{.3em}}}{}

\newtheorem{definition}{Definition}[section]
\newtheorem{theorem}[definition]{Theorem}
\newtheorem{lemma}[definition]{Lemma}

\newtheorem{corollary}[definition]{Corollary}

\newcounter{kpremark}
\newcounter{digressionitem}
\newenvironment{digressionresult}[1]{\addtocounter{digressionitem}{1}
\leftskip.25in
\rightskip\leftskip
(\thedigressionitem)\quad #1}

\newcommand{\mod}[1]{\ensuremath{\hspace{.5em}(#1)}}

\newcommand{\newo}{\ensuremath{\mathcal{O}}}

\begin{document}

\begin{frontmatter}






\title{A Hecke algebra attached to mod 2 modular forms of level 3}
\author{Paul Monsky}

\address{Brandeis University, Waltham MA  02454-9110, USA. monsky@brandeis.edu}

\begin{abstract}
Let $D$ in $Z/2[[x]]$ be $\sum x^{n^{2}}$, $n>0$ and prime to $6$. Let $W$ be spanned by the $D^{k}$, $k>0$ and prime to $6$. Then the formal Hecke operators $T_{p}$, $p>3$, stabilize $W$, and it can be shown that they act locally nilpotently. We show that the completion of the Hecke algebra generated by these $T_{p}$ acting on $W$, with respect to the maximal ideal generated by the $T_{p}$,  is a power series ring in $T_{7}$ and $T_{13}$ with an element of square $0$ adjoined. This may be viewed as a level 3 analog of the level 1 results of Nicolas and Serre---the Hecke stable space they study is spanned by the odd powers of the mod $2$ reduction of $\Delta$, and their resulting completed Hecke algebra is a power series ring in $T_{3}$ and $T_{5}$. In a digression appearing in section \ref{section4} we sketch a new and simpler proof of the results of Nicolas and Serre.

\end{abstract}


\end{frontmatter}


\section{Introduction}
\label{section1}

For each odd prime $p$ we have a formal Hecke operator $T_{p} : Z/2[[x]]\rightarrow Z/2[[x]]$ taking $\sum c_{n}x^{n}$ to $\sum c_{pn}x^{n} + \sum c_{n}x^{pn}$; these operators commute.  Let $F$ be $\sum x^{n^{2}}$, $n$ odd and $> 0$, and let $V$ be the subspace of $Z/2[[x]]$ spanned by $F$, $F^{3}$, $F^{5}$, $F^{7}$, $\ldots$. Nicolas and Serre \cite{3}, \cite{4} have shown that each $T_{p}$ stabilizes $V$, and they have analyzed the action on $V$ of the algebra generated by the  maps $T_{p} : V\rightarrow V$. In their analysis they make $V$ into a $Z/2[[X,Y]]$-module with $X$ and $Y$ acting by $T_{3}$ and $T_{5}$. They show that under this action $V$ is the Matlis dual of $Z/2[[X,Y]]$, and that each $T_{p} : V\rightarrow V$ is a power series with zero constant term in $X = T_{3}$ and $Y = T_{5}$.

$V$ is in fact the space of ``mod 2 modular forms of level 1'' (or more accurately the set of odd power series in that space), and the Nicolas-Serre results are important in the study of such forms. These results depend on two technical lemmas, stated as Propositions 4.3 and 4.4 of \cite{3}, with unpublished proofs. Gerbelli-Gauthier \cite{1} has found a sensible proof of Proposition 4.3, and I've made further simplifications and generalizations in \cite{2}. In a digression appearing in section \ref{section4} of this note I sketch a new and simpler proof of the Nicolas and Serre results that avoids the use of Proposition 4.4 of \cite{3}.

I'll now describe the results of this note, which are related to the Hecke action on the space of ``mod 2 modular form for $\Gamma_{0}(3)$''.

\begin{definition}
\label{def1.1}
$D$ in $Z/2[[x]]$ is $\sum x^{n^{2}}$, where $(n,6)=1$ and $n>0$. $W$ is spanned by the $D^{k}$ with $(k,6)=1$ and $k>0$. $W1$ and $W5$ are the subspaces of $W$ spanned by the $D^{k}$ with $k\equiv 1 \mod 6$ and $k\equiv 5 \mod 6$ respectively.
\end{definition}

Let $G=F(x^{3})$; $G$ is a mod 2 modular form for $\Gamma_{0}(3)$. In the next section we'll show that $F^{4}+G^{4}=FG$, and that the space $M(\mathit{odd})$ of odd mod 2 modular forms for $\Gamma_{0}(3)$ is spanned by the $F^{i}G^{j}$ with $i+j$ odd. Using a Hecke-stable filtration of $M(\mathit{odd})$, we'll interpret $W$ as a subquotient of $M(\mathit{odd})$, and use this interpretation to show that if $p>3$ then $T_{p}(D^{m})$ is a sum of $D^{k}$ with $k\le m$ and $k\equiv pm \mod {24}$.  This will allow us to make $W1$ and $W5$ into $Z/2[[X,Y]]$-modules with $X$ and $Y$ acting by $T_{7}$ and $T_{13}$. Corollary 4.2 of \cite{2} then leads to an analog, for the action of $X$ on $W5$, to Proposition 4.3 of \cite{3}. In section \ref{section3} we turn to the
theory of binary theta series, using it to construct a ``dihedral space'', contained in $W5$ and annihilated by $X$, on which the action of $Y$ can be easily described.

In the final sections we use the results above, first establishing results for $W5$, $T_{7}$ and $T_{13}$ completely analogous to the Nicolas-Serre results for $V$, $T_{3}$ and $T_{5}$. There's a surprise though when we turn to $W$, which is stabilized by all the $T_{p}$, $p>3$. Now the ``completed Hecke algebra'' is no longer a 2-variable power series ring. Indeed it contains nilpotents, and is isomorphic to $Z/2[[X,Y]]$ with an element of square 0 adjoined.  This element of square 0 may be written as $T_{5} + \lambda (T_{7},T_{13})$ where $\lambda$ in $Z/2[[X,Y]]$ is $X+Y+\mbox{ higher degree terms}$. Each $T_{p}$ with $p\equiv 1\mod 6$ is a power series with zero constant term in $X=T_{7}$ and $Y=T_{13}$, while each $T_{p}$ with $p\equiv 5\mod 6$ is the composition of $T_{5}$ with a power series in $T_{7}$ and $T_{13}$.

\section{Preliminaries in level 3}
\label{section2}

$F=\sum_{\, n\ \mathrm{odd},\ n>0} x^{n^{2}}$ and $G=F(x^{3})$ are the elements of $Z/2[[x]]$ in the introduction.\\

\begin{definition}
\label{def2.1}
$H=F(x^{9})$. $E=\sum_{\, (n,3)=1,\ n>0} x^{n^{2}}$.
\end{definition}

Note that the $D$ of Definition \ref{def1.1} is just $F+H$, and that $E^{4}+E=D$.

\begin{definition}
\label{def2.2}
$M_{k}\subset Z/2[[x]]$ consists of all $f$ such that there is a modular form of weight $k$ for $\Gamma_{0}(3)$ whose expansion at the cusp $\infty$ lies in $Z[[x]]$, (we write $x$ in place of the more customary $q$ throughout), and reduces mod $2$ to~$f$.
\end{definition}

Now $\Gamma_{0}(3)$ has two cusps, 0 and $\infty$. It follows that the weight 2 Eisenstein space and the subspace of the weight 4 Eisenstein space consisting of forms vanishing at $\infty$ each have complex dimension 1. Let $P=1+\cdots$ and $B=x+\cdots$ be the normalized generators of these spaces. Classical formul\ae\ show that $P$ and $B$ lie in $Z[[x]]$ with mod 2 reductions 1 and $r$, where $r=\sum_{\, n>0} (x^{n^{2}}+x^{2n^{2}}+x^{3n^{2}}+x^{6n^{2}})$. Note that $r^{2}+r=F+G$. Using multiplication by $P$ we see that $M_{0}\subset M_{2}\subset M_{4}\subset M_{6}\cdots$.

\begin{definition}
\label{def2.3}
$M$ is the union of the $M_{k}$.
\end{definition}

Note that $M$ is a subring of $Z/2[[x]]$.  I'll now write down an explicit basis of $M_{6k}$ over $Z/2$. Classical formul\ae\ for the dimensions of spaces of modular forms tell us that $M_{6k}$ has dimension $2k+1$.

\begin{theorem}
\label{theorem2.4}
$\{1,r,r^{2}\}$ is a basis for $M_{6}$.
\end{theorem}

\begin{proof}
$1\in M_{0}\subset M_{6}$. Since $PB$ reduces to $r$, $r$ is in $M_{6}$. We show that $r^{2}$ is in $M_{6}$ as follows. Let $C=\eta(z)^{6}\eta(3z)^{6}$ be the normalized weight 6 newform for $\Gamma_{0}(3)$. Comparing the first few coefficients we see that the weight 8 modular forms $P\cdot\left(\frac{PB-C}{27}\right)$ and $B^{2}$ are equal. So the weight 6 form $\frac{PB-C}{27}=\frac{B^{2}}{P}$ has an expansion that lies in $Z[[x]]$ and reduces to $r^{2}$. Since $M_{6}$ has dimension 3, the theorem follows.
\qed
\end{proof}

\begin{corollary}
\label{corollary2.5}
A basis of $M_{6k}$ is given by the $r^{i}$ with $0\le i\le 2k$.
\end{corollary}

\begin{definition}
\label{def2.6}
An element of $Z/2[[x]]$ is odd if it lies in $x\cdot Z/2[[x^{2}]]$. $M_{k}(\mathit{odd})$ and $M(\mathit{odd})$ are the subspaces of $M_{k}$ and $M$ consisting of odd elements. 
\end{definition}

Note that $r+r^{2}=F+G$ is odd.  Corollary \ref{corollary2.5} gives:

\begin{theorem}
\label{theorem2.7}
A basis of $M_{6k}(\mathit{odd})$ is given by the $r^{2i}(r+r^{2})$ with $0\le i\le k-1$. The $r^{2i}(r+r^{2})$ are a basis of $M(\mathit{odd})$.
\end{theorem}

\begin{theorem}
\label{theorem2.8}
$F=r+r^{2}+r^{3}+r^{4}$, and $G=r^{3}+r^{4}$.
\end{theorem}

\begin{proof}
$F$ and $G$, being the mod $2$ reductions of the expansions at $\infty$ of the weight 12 modular forms $\Delta (z)$ and $\Delta (3z)$ for $\Gamma_{0}(3)$, lie in $M_{12}$. As they are in $M_{12}(\mathit{odd})$, they are linear combinations of $r+r^{2}$ and $r^{3}+r^{4}$. Since $G=x^{3}+\cdots$, $G=r^{3}+r^{4}$. Finally, $F+G=r+r^{2}$.
\qed
\end{proof}

\begin{theorem}
\label{theorem2.9} \hspace{1em}\\
\vspace{-5ex}
\begin{enumerate}
\item[(a)] $F^{4}+FG+G^{4}=0$
\item[(b)] $D^{3}=G$
\end{enumerate}

\end{theorem}

\begin{proof}
$FG=(r+r^{2}+r^{3}+r^{4})(r^{3}+r^{4})=(r+r^{2})^{4}=F^{4}+G^{4}$, giving (a). Replacing $x$ by $x^{3}$ we find that $G^{4}+GH+H^{4}=0$. Adding this to (a) we see that $(F^{4}+H^{4})+G(F+H)=0$. So $G=(F+H)^{3}=D^{3}$.
\qed
\end{proof}

\begin{theorem}
\label{theorem2.10}
$M(\mathit{odd})\subset Z/2[F,G]$.
\end{theorem}

\begin{proof}
Let $A_{n}=r^{2n}(r+r^{2})$. By Theorem \ref{theorem2.7} it's enough to show that the $A_{n}$ are in $Z/2[F,G]$. Now $A_{0}=F+G$, while $A_{1}=G$. Since $A_{n+2}+A_{n+1}=(r^{2}+r^{4})A_{n}=(F+G)^{2}A_{n}$, an induction gives the theorem.
\qed
\end{proof}


\begin{theorem}
\label{theorem2.11}
Viewed as a $Z/2[G^{2}]$-module, $M (\mathit{odd})$ is free of rank 4; a basis is given by $\{G, F, F^{2}G, F^{3}\}$.
\end{theorem}

\begin{proof}
$F^{4}+ FG +G^{4}=0$. So $F$ has degree 4 over $Z/2(G)$, and $Z/2[F,G]$ is a free rank 8 $Z/2[G^{2}]$-module; a basis is given by ${1, F, F^{2}, F^{3},G, FG, F^{2}G, F^{3}G}$. Since $G$, $F$, $F^{2}G$ and $F^{3}$ are odd, while the other 4 elements in the basis are ``even'', the result follows from Theorem \ref{theorem2.10}.
\qed
\end{proof}

\begin{definition}
\label{def2.12}
For $i$ in $\{0, 1, 2\}$, $p_{3,i}: Z/2[[x]]\rightarrow Z/2[[x]]$ is the $(Z/2[G]-\mbox{linear})$ map taking $\sum c_{n}x^{n}$ to $\sum_{\, n\equiv i \mod{3}}c_{n}x^{n}$.
\end{definition}

\begin{lemma}
\label{lemma2.13} \hspace{1em}\\
\vspace{-5ex}
\begin{enumerate}
\item[(a)] \parbox{.5\textwidth}{$p_{3,1}(F^{2}G)=0$} $p_{3,2}(F)=0$
\item[(b)] \parbox{.5\textwidth}{$p_{3,1}(F)=D$} $p_{3,2}(F^{2}G)=D^{5}$
\item[(c)] \parbox{.5\textwidth}{$p_{3,1}(r^{2})=E^{4}$} $p_{3,2}(r^{2})=E^{2}$
\item[(d)] \parbox{.5\textwidth}{$p_{3,1}(F^{3})=E^{16}D^{3}$} $p_{3,2}(F^{3})=E^{8}D^{3}$
\end{enumerate}
\end{lemma}

\begin{proof}
Since all squares are 0 or 1 mod 3, $p_{3,2}(F)=0$, and $p_{3,1}(F^{2}G)=Gp_{3,1}(F^{2})=0$. For the same reason, $p_{3,1}(F)=D$, and so $p_{3,2}(F^{2}G)=G\cdot p_{3,2}(F^{2})=D^{3}\cdot D^{2}=D^{5}$. Now $r=\sum_{\, n>0}(x^{n^{2}} + x^{2n^{2}} + x^{3n^{2}} + x^{6n^{2}} )$. Since all squares are 0 or 1 mod 3, 
$p_{3,1}(r^{2})$ and $p_{3,2}(r^{2})$ are $\sum_{\, (n,3)=1,\ n>0} x^{4n^{2}}=E^{4}$, and $\sum_{\, (n,3)=1,\ n>0} x^{2n^{2}}=E^{2}$ respectively. Next note that $(F+G)^{3}=(r+r^{2})^{3}=r^{3}+r^{4}+r^{5}+r^{6}=G+r^{2}G$. So $F^{3}=F^{2}G+FG^{2}+G^{3}+G+r^{2}G$. Applying $p_{3,1}$ to the right-hand side gives $0+G^{2}D+0+0+GE^{4}=D^{3}(D^{4}+E^{4})$. Since $D+E=E^{4}$, $p_{3,1}(F^{3})=E^{16}D^{3}$. Similarly we find that $p_{3,2}(F^{3})=GD^{2}+0+0+0+GE^{2}=G(D^{2}+E^{2})=E^{8}D^{3}$.
\qed
\end{proof}

\begin{definition}
\label{def2.14}
$K1$ is the $Z/2[G^{2}]$-submodule of $M (\mathit{odd})$ consisting of the $f$ annihilated by $p_{3,2}$. $K5$ is the $Z/2[G^{2}]$-submodule of $M(\mathit{odd})$ consisting of the $f$ annihilated by $p_{3,1}$.
\end{definition}

\begin{theorem}
\label{theorem2.15} \hspace{1em}\\
\vspace{-5ex}
\begin{enumerate}
\item[(a)] $F$ and $G$ are a $Z/2[G^{2}]$-basis of $K1$.
\item[(b)] $F^{2}G$ and $G$ are a $Z/2[G^{2}]$-basis of $K5$.
\end{enumerate}
\end{theorem}

\begin{proof}
$F$ and $G$ lie in $K1$. Theorem \ref{theorem2.11} shows that to prove (a) it's enough to show that any $Z/2[G^{2}]$-linear combination of $F^{2}G$ and $F^{3}$ annihilated by $p_{3,2}$ is $0$, or in other words that $p_{3,2}(F^{2}G)$ and $p_{3,2}(F^{3})$ are linearly independent over $Z/2[G^{2}]=Z/2[D^{6}]$. But by Lemma \ref{lemma2.13}, $p_{3,2}(F^{3})/p_{3,2}(F^{2}G)=E^{8}/D^{2}$, and since $E^{4}+E=D$, this quotient is not even in $Z/2(D)$. Similarly, $p_{3,1}(F^{3})/p_{3,1}(F)=E^{16}D^{2}$, which is not
in $Z/2(D)$, let alone in $Z/2(G^{2})$, and we get (b).
\qed
\end{proof}

\begin{theorem}
\label{theorem2.16}
If $(n,6)=1$ and $p>3$, $T_{p}(D^{n})$ is a sum of $D^{k}$ with $k\equiv pn \mod{6}$.
\end{theorem}

\begin{proof}
Definition \ref{def2.14} shows that the $T_{p}$ with $p\equiv 1\mod{6}$ stabilize $K1$ and $K5$. Suppose then that $p\equiv 1 \mod{6}$. Then for each $m$, $FG^{2m}$ lies in $K1$. So $T_{p}(FG^{2m})$ lies in $K1$ and, by Theorem \ref{theorem2.15}(a), is a sum of various $(F)G^{2r}$ and various $(G)G^{2s}$. Applying $p_{3,1}$, we find that $T_{p}(D^{6m+1})$ is a sum of various $D^{6r+1}$. Also, $(F^{2}G)G^{2m}$ lies in $K5$. So $T_{p}\left((F^{2}G)G^{2m}\right)$ lies in $K5$, and is a sum of various $(F^{2}G)G^{2r}$ and various $(G)G^{2s}$. Applying $p_{3,2}$ we find that $T_{p}(D^{6m+5})$ is a sum of various $D^{6r+5}$. The argument when $p\equiv 5\mod{6}$ is similar, but now we use the fact, evident from Definition \ref{def2.14}, that $T_{p}(K1)\subset K5$ while $T_{p}(K5)\subset K1$.
\qed
\end{proof}

\begin{theorem}
\label{theorem2.17}
If $(n,6)=1$ and $p>3$, $T_{p}(D^{n})$ is a sum of $D^{k}$ with $k\equiv pn\mod{24}$, $k\le n$.
\end{theorem}

\begin{proofsketch}
Replace $\Gamma_{0}(3)$ by $\Gamma_{0}(9)$ in Definition \ref{def2.2} and denote the resulting $M_{k}$ by $L_{k}$. Classical formul\ae\ show that the dimension of $L_{2k}$ is $2k+1$. Furthermore, $\Gamma_{0}(9)$ has 4 cusps, and a computation
with Eisenstein series shows that $L_{2}$ has $\{1, E, E^{2}\}$ as a basis. If follows that the $E^{i}$ with $0\le i \le 4n$ are a basis of $L_{4n}$. Now write $T_{p}(D^{n})$ as a sum of distinct $D^{k}$. Since each exponent appearing in $D$ is $\equiv 1\mod{24}$, each exponent appearing in $T_{p}(D^{n})$ is $\equiv pn\mod{24}$. So if we restrict the sum to those $D^{k}$ with $k\equiv pn\mod{24}$ we still get $T_{p}(D^{n})$. Furthermore, $T_{p}\left((E^{4}+E)^{n}\right)$ is the sum of the corresponding $(E^{4}+E)^{k}$. Since $(E^{4}+E)^{k}$ lies in $L_{4n}$, and $T_{p}$ stabilizes $L_{4n}$, the sum of the $(E^{4}+E)^{k}$ has degree $\le 4n$ in $E$ and each $k\le n$.
\end{proofsketch}

\begin{remark*}{Remark}
I'll give a slightly different interpretation of $W$. Let $N_{2}$ and $N_{1}$ be the $Z/2[G^{2}]$-submodules of $M(\mathit{odd})$ having $\{G, F, F^{2}G\}$ and $\{G\}$ as bases. Since the $T_{p}$, $p>3$, stabilize $V$, they stabilize $N_{1}$. By Theorem \ref{theorem2.15}, $N_{2}=K1+K5$, and so these $T_{p}$ stabilize $N_{2}$ as well. (This can also be seen using the fact that $N_{2}$ consists of those elements of $M(\mathit{odd})$ whose traces from $Z/2(F,G)$ to $Z/2(G)$ are 0.)  Using Theorem \ref{theorem2.15} and Lemma \ref{lemma2.13}, we see that the map $pr = p_{3,1}+p_{3,2}$ maps $N_{2}$ onto $W$ with kernel $N_{1}$. So $W$ identifies with $N_{2}/N_{1}$, and since $pr$ commutes with the $T_{p}$, the identification preserves the Hecke action. $W$ should be thought of as the ``new part'' of $M(\mathit{odd})$. Namely we have a Hecke-stable filtration, $M(\mathit{odd})\supset N_{2} \supset N_{1} \supset (0)$, and $N_{1}$ and $M(\mathit{odd})/N_{2}$ are essentially the Hecke-module studied
by Nicolas and Serre. (The trace map $Z/2[F,G]\rightarrow Z/2[G]$, which takes $F^{3}$ to $G$, gives a Hecke-isomorphism of $M(\mathit{odd})/N_{2}$ with $N_{1}$.)
\end{remark*}

We now take up the problem of calculating the $T_{7}(D^{m})$ with $m\equiv 5\mod{6}$, starting with ``initial values''.

\begin{lemma}
\label{lemma2.18}
$T_{7}$ takes $D^{5}$, $D^{11}$, $D^{17}$, $D^{23}$, $D^{29}$, $D^{35}$, $D^{41}$ and $D^{47}$ to $0$, $D^{5}$, $0$, $D^{17}$, $D^{11}$, $D^{29}+D^{5}$, $D^{23}$ and $D^{41}+D^{17}$ respectively.
\end{lemma}

\begin{proof}
I'll illustrate by calculating $T_{7}(D^{35})$.By Theorem \ref{theorem2.17}, this is a linear combination of $D^{5}$ and $D^{29}$. Now $D^{5}=D\cdot D^{4}=(x + x^{25} + \cdots)\cdot(x^{4}+x^{100}+\cdots) = x^{5}+x^{29}+\cdots $ while $D^{29} = x^{29}+\cdots $. Furthermore, $D^{35}= D^{32}\cdot G = (x^{32}+x^{800}+\cdots)\cdot G$. Since the coefficients of $x^{3}$ and $x^{171}$ in $G$ are 1 and 0, the coefficients of $x^{35}$ and $x^{203}$ in $D^{35}$ are 1 and 0. So $T_{7}(D^{35})=x^{5}+0\cdot x^{29}+\cdots$, and can only be $D^{5}+D^{29}$.  The other results are proved similarly.
\qed
\end{proof}

To handle $T_{7}(D^{m})$ with $m\equiv 5\mod{6}$ and otherwise arbitrary, we develop a recursion formula similar to those that appear in \cite{3}. Namely let $V$ now denote the subspace of $Z/2[t]$ spanned by the $t^{k}$ with $k$ odd. If $n$ is odd $3n+2\equiv 5\mod{6}$, and so  $T_{7}(D^{3n+2})$ is a sum of $D^{k}$ with $k\equiv 5\mod{6}$,
and so can be written as $D^{2}\cdot A_{n}(D^{3})$ for some $A_{n}$ in $V$.  Lemma \ref{lemma2.18} tells us that $A_{1}$, $A_{3}$, $A_{5}$, $A_{7}$, $A_{9}$, $A_{11}$, $A_{13}$ and $A_{15}$ are $0$, $t$, $0$, $t^{5}$, $t^{3}$, $t^{9}+t$, $t^{7}$ and $t^{13}+t^{5}$ respectively. We'll complete the calculation of the $A_{n}$, $n$ odd, by showing that $A_{n+16}=t^{16}A_{n}+t^{4}A_{n+4}+t^{2}A_{n+2}$.

\begin{lemma}
\label{lemma2.19}
For $u$ in $Z/2[[x]]$, $T_{7}(G^{16}u)=G^{16}T_{7}(u)+G^{4}T_{7}(uG^{4})+G^{2}T_{7}(uG^{2})$.
\end{lemma}

\begin{proof}
Let $U(X,Y)$ be $X^{8}+Y^{8}+X^{2}Y^{2}+XY$. Then $U\left(F(x^{7}), F(x)\right)=0$; this is the ``modular equation of level 7 for $F$'', analogous to the polynomial relation between $F$ and $G$. Replacing $x$ by $x^{3}$ we find that $U\left(G(x^{7}), G(x)\right)=0$. Now let $L$ be an algebraic closure of $Z/2$. We have 8 imbeddings $\varphi_{i} : Z/2[[x]]\rightarrow L[[x^{1/7}]]$, the first of which takes $f$ to $f(x^{7})$, while each of the others takes $f$ to $f(\lambda x^{1/7})$ for some $\lambda$ in $L$ with $\lambda^{7}=1$.  Replacing $x$ by $\lambda x^{1/7}$ in the identity $U\left(G(x^{7}), G(x)\right)=0$, and using the symmetry of $U$, we find that $U\left(\varphi_{i}(G), G\right)=0$ for each $\varphi_{i}$. The definition of $T_{7}$ shows that for each $f$ in $Z/2[[x]]$, $T_{7}(f)=\sum\varphi_{i}(f)$. The identity $U\left(\varphi_{i}(G), G\right)=0$ tells us that $\varphi_{i}(G^{16})=G^{16}\varphi_{i}(1)+G^{4}\varphi_{i}(G^{4})+G^{2}\varphi_{i}(G^{2})$. Multiplying the $i$\textsuperscript{\,th} of these equations by $\varphi_{i}(u)$ and summing gives the result.
\qed
\end{proof}

\begin{theorem}
\label{theorem2.20}
Let $A_{n}$ in $Z/2[t]$, $n$ odd, be the polynomial for which the identity $T_{7}(D^{3n+2})=D^{2}\cdot A_{n}(D^{3})$ holds. Then $A_{n+16}=t^{16}A_{n}+t^{4}A_{n+4}+t^{2}A_{n+2}$.
\end{theorem}

\begin{proof}
Take $u=D^{2}G^{n}=D^{3n+2}$ in Lemma \ref{lemma2.19}. The left hand side of the equation there is $T_{7}(D^{3(n+16)+2})=D^{2}A_{n+16}(G)$. Similarly, the right hand side is $D^{2}\left(G^{16}A_{n}(G)+G^{4}A_{n+4}(G)+G^{2}A_{n+2}(G)\right)$. So $A_{n+16}(G)=G^{16}A_{n}(G)+G^{4}A_{n+4}(G)+G^2A_{n+2}(G)$, and the theorem follows.
\qed
\end{proof}

We now introduce some notation from \cite{2}.

\begin{definition}
\label{def2.21} \hspace{1em}\\
\vspace{-5ex}
\begin{enumerate}
\item[(1)] $g : N\rightarrow N$ is the unique function for which $g(2n) = 4g(n)$, and $g(2n+1)= g(2n)+1$. Note that if $n$ is a sum of distinct powers of 2, then $g(n)$ is the sum of the squares of the summands.
\item[(2)] Let $t$ be an indeterminate over $Z/2$, and $V\subset Z/2[t]$ be spanned by the $t^{n}$ with $n$ odd. The last sentence of (1) shows that $(a,b)\rightarrow t^{1+2g(a)+4g(b)}$ is a 1--1 map between $N\times N$ and the monomials in $V$. We denote the image
of $(a,b)$ by $[a,b]$.
\item[(3)] We put a total ordering on the monomials in $V$ as follows. $[c,d]\prec [a,b]$ if $c+d < a+b$ or if $c+d = a+b$ and $d<b$. If $[c,d]\prec [a,b]$ we say $[c,d]$ is ``earlier'' than $[a,b]$.
\end{enumerate}
\end{definition}

To illustrate, the above ordering begins: $[0,0]\prec[1,0]\prec[0,1]\prec[2,0]\prec[1,1]\prec[0,2]\prec[3,0]\prec\cdots$, i.e.\ $t\prec t^{3}\prec t^{5}\prec t^{9}\prec t^{7}\prec t^{17}\prec t^{11}\prec\cdots$.

\begin{theorem}
\label{theorem2.22}
Suppose we have $A_{n}$, $n$~odd in $V$ defined by the recursion $A_{n+16}=t^{16}A_{n}+t^{4}A_{n+4}+t^{2}A_{n+2}$ where the initial values $A_{1}$, $A_{3}$, $A_{5}$, $A_{7}$, $A_{9}$, $A_{11}$, $A_{13}$, $A_{15}$ are $0$, $t$, $0$, $t^{5}$, $t^{3}$, $t^{9}+t$, $t^{7}, t^{13}+t^{5}$. Then if $t^{n}=[a,b]$, $A_{n}$ is a sum of $[c,d]$ with $c+d<a+b$. Furthermore if $a>0$, $A_{n}=[a-1,b]+ \mbox{a sum of earlier monomials}$.
\end{theorem}

\begin{proof}
This is precisely Corollary 4.2 of \cite{2}; the proof given there is motivated by Gerbelli-Gauthier \cite{1}.
\qed
\end{proof}

Now let $W$, $W1$ and $W5$ be as in Definition \ref{def1.1}. Just as Corollary 3.8 of \cite{2} gives rise to Proposition 4.3 of \cite{3} (see the digression following Theorem \ref{theorem4.17} of the present paper), so does the above Theorem \ref{theorem2.22} translate into a result about the action of $T_{7}$ on $W5$.


\begin{definition}
\label{def2.23}
$[a,b,G]$ in $W5$ is the image of $[a,b]$ in $V$ under the map $V\rightarrow W5$ taking $f$ to $D^{2}\cdot f(G)=D^{2}\cdot f(D^{3})$.
\end{definition}

Since the $D^{6n+5}$ are a monomial basis of $W5$, and the $[a,b]$ run over all the $t^{n}$ with $n$ odd, the $[a,b,G]$ are a monomial basis of $W5$. Explicitly, $[a,b,G]=D^{n}$ where $n=5+6g(a)+12g(b)$. The ordering of the $[a,b]$ gives an ordering of the $[a,b,G]$ starting out: $D^{5}\prec D^{11}\prec D^{17}\prec D^{29}\prec D^{23}\prec D^{53}\prec \cdots$.

\begin{theorem}
\label{theorem2.24}
$T_{7}\left([a,b,G]\right)$ is a sum of $[c,d,G]$ with $c+d<a+b$. If $a>0$, $T_{7}\left([a,b,G]\right)=[a-1,b,G]+$ a sum of earlier monomials in $D$.
\end{theorem}

\begin{proof}
Identify $V$ with $W5$ by the map $f\rightarrow D^{2}f(G)$. Then $T_{7}$ may be viewed as a map $V\rightarrow V$. Since $T_{7}(D^{3n+2})=D^{2}A_{n}(G)$ with $A_{n}$ as in Theorem \ref{theorem2.20}, $T_{7}$ takes $t^{n}$ to $A_{n}$ for $n$ odd. We have seen that these $A_{n}$ satisfy the recursion and initial conditions of Theorem \ref{theorem2.22}. Applying Theorem \ref{theorem2.22} and passing back from $V$ to $W5$ we get the result.
\qed
\end{proof}

\section{Spaces attached to Gauss-classes of ideals in $\bm{Z[i]}$}
\label{section3}

Fix a power, $q$, of 2. In this section we (essentially) use binary quadratic forms of discriminant $-64q^{2}$ to construct a subspace, $DI(q)$, of $W5$ of dimension $q$, stable under the $T_{p}$, $p\equiv 1\mod{6}$, and annihilated by the $T_{p}$, $p\equiv 7\mod{12}$. And we'll give a simple description of the action of $T_{p}$, $p\equiv 1\mod{12}$, on $DI(q)$, basically involving Gaussian composition of forms.

$DI(q)$ will be the image under $p_{3,2}$ of the mod 2 reduction of a certain additive subgroup of $Z[[x]]$. This subgroup consists of expansions at infinity of certain weight 1 modular forms of level a power of 2, and is stable under the action of a ``characteristic 0, weight 1'' Hecke algebra. In defining the subgroup and analyzing the Hecke action, we will however avoid the classical language of binary theta series and Gaussian composition, and instead work with a certain equivalence relation on the ideals $I$ of $Z[i]$ of odd norm. We call this relation ``Gauss-equivalence'', and the corresponding equivalence classes ``Gauss-classes''.

Each $I$ of odd norm has a generator $a+2bi$ with $a$ odd, and this generator is unique up to multiplication by $\pm 1$. When we speak of an ideal $(\alpha)$, we'll assume $\alpha$ has the above form.

\begin{definition}
\label{def3.1}
$(\alpha)$ and $(\beta)$ are Gauss-equivalent if there is an integer $N$ such that, in $Z[i]$, $N\alpha \equiv \beta\mod{4q}$.
\end{definition}

Evidently $N$ is odd, and Gauss-equivalence is an equivalence relation. If $R_{1}$ and $R_{2}$ are Gauss-classes, all ideals that are products of an element of $R_{1}$ and an element of $R_{2}$ are Gauss-equivalent. The semi-group of Gauss-classes that we get in this way is a group, with the inverse of the class of $(\alpha)$ being the class of $(\bar{\alpha})$. We call this group the Gauss group.

\begin{theorem}
\label{theorem3.2}
The Gauss group is cyclic of order $2q$. If the norm of $I$ is $5 \bmod{8}$, $I$ generates the Gauss group.
\end{theorem}

\begin{proof}
The $2q$ ideals $(1+2bi)$, $0\le b< 2q$ are obviously inequivalent. Suppose $I=(a+2bi)$. Choose $N$ in $Z$ so that $Na\equiv 1\mod{4q}$. Then $N\cdot(a+2bi)$ is congruent, mod $4q$, to some $1+2ci$, with $0\le c < 2q$, and so there are just $2q$ Gauss-classes. Suppose $I=(\alpha)$ has norm $\equiv 5\mod{8}$.  Then $\alpha =(\mathit{odd})+2i(\mathit{odd})$. So $\alpha^{2}=(\mathit{odd})+4i(\mathit{odd})$, $\alpha^{4}=(\mathit{odd})+8i(\mathit{odd})$, $\ldots$, and in particular $\alpha^{q}=(\mathit{odd})+2qi(\mathit{odd})$, $\alpha^{2q}=(\mathit{odd})+4qi(\mathit{odd})$. So $(\alpha^{q})$ isn't equivalent to $(1)$, and the class of $(\alpha)$ generates the group.
\qed
\end{proof}

Now the ideals in the principal Gauss-class are just the $(a+4bqi)$ with $a$ odd. The proof of Theorem \ref{theorem3.2} shows:

\begin{corollary}
\label{corollary3.3}
There is a unique Gauss-class $R$ of order 2; the ideals of $R$ are the $(a+2bqi)$ with $a$ and $b$ odd.  (We'll denote this class of order 2 by $\mathit{AMB}$.)
\end{corollary}

\begin{definition}
\label{def3.4}
If $R$ is a Gauss-class, then $\theta(R)$ in $Z[[x]]$ is $\sum x^{(\mathrm{norm}\ I)}$, where $I$ runs over the ideals in $R$.
\end{definition}

\begin{remark*}{Remark}
Each $\theta(R)$ is in fact the expansion at infinity of a weight 1 modular form of level a power of 2, with character $n\rightarrow (-1/n)$, but we won't explicitly use this fact. However it motivates:
\end{remark*}

\begin{definition}
\label{def3.5}
Let $p$ be an odd prime. Then $T_{p} : Z[[x]]\rightarrow Z[[x]]$ is the map $\sum c_{n}x^{n}\rightarrow \sum c_{pn}x^{n}+(-1/p)\cdot \sum c_{n}x^{pn}$. Note that the mod $2$ reduction of this $T_{p}$ is the $T_{p} : Z/2[[x]]\rightarrow Z/2[[x]]$ of the introduction.
\end{definition}

Now let $p$ be a prime $\equiv 1\mod{4}$. Then there are just two ideals of norm $p$ in $Z[i]$. Call these $P$ and $\bar{P}$.

\begin{theorem}
\label{theorem3.6}
$T_{p}$ takes $\theta(R)$ to $\theta(PR)+\theta(\bar{P}R)$.
\end{theorem}

\begin{proof}
Suppose first that $P$ and $\bar{P}$ are inequivalent, so that $PR$ and $\bar{P}R$ are distinct classes. Fix an odd integer $n$. The ideals lying in $R$ whose norm is either $n$ or $p^{2}n$ are of four types:
\begin{enumerate}
\item[(1)] Ideals of norm $p^{2}n$ that are prime to $P$.
\item[(2)] Ideals of norm $p^{2}n$ that are prime to $\bar{P}$.
\item[(3)] Ideals of norm $n$.
\item[(4)] Ideals of norm $p^{2}n$ divisible by $P\bar{P}=(p)$.
\end{enumerate}

Denote the number of ideals of types (1), (2) and (3) by $r_{1}$, $r_{2}$ and $r_{3}$.  There are evidently $r_{3}$ ideals of type (4). The coefficients of $x^{n}$ and $x^{p^{2}n}$ in $\theta(R)$ are then $r_{3}$ and $r_{1}+r_{2}+r_{3}$. So the coefficient of $x^{pn}$ in $T_{p}(\theta(R))$ is $r_{1}+r_{2}+2r_{3}$.

Now the ideals of norm $pn$ in the class (that includes the elements of) $PR$ are of two types:

\begin{enumerate}
\item[(a)] Ideals prime to $P$
\item[(b)] Ideals divisible by $P$
\end{enumerate}

$I\rightarrow \bar{P}I$ sets up a 1--1 correspondence between the type (a) ideals and the type (1) ideals of $R$, while $I\rightarrow I/P$ sets up a 1--1 correspondence between the type (b) ideals and the type (3) ideals of $R$. So the coefficient of $x^{pn}$ in $\theta(PR)$ is $r_{1}+r_{3}$, and similarly the coefficient of $x^{pn}$ in $\theta(\bar{P}R)$ is $r_{2}+r_{3}$. So the coefficients of $x^{pn}$ in $T_{p}(\theta(R))$ and in $\theta(PR)+\theta(\bar{P}R)$ are equal. A similar but simpler argument works for the coefficients of $x^{n}$ when $n$ is prime to $p$.

Suppose finally that $P$ and $\bar{P}$ are equivalent, so that $\theta(PR)=\theta(\bar{P}R)$. We want to show that $T_{p}(\theta(R))=2\theta(PR)$. Fix an ideal $J$, lying in $R$ or $PR$, and prime to $P$. We restrict our attention to ideals lying in $R$ or $PR$ whose ``prime to $p$ part'' is $J$. It's easy
to see that there is an integer $k$ such that the number of such ideals of norm $n$ (resp.\ $p^{2}n$) lying in $R$ is $k$ (resp.\ $k+2$), while the number of ideals of norm $pn$ of this form lying in $PR$ is $k+1$. Since $k+(k+2) = 2(k+1)$, the ideals of this form give the same contributions to the coefficients of $x^{pn}$ in $T_{p}(\theta(R))$ and $2\theta(PR)$; summing over $J$ we complete the proof.
\qed
\end{proof}

\begin{theorem}
\label{theorem3.7}
If $p\equiv 3\mod{4}$, $T_{p}(\theta(R))=0$.
\end{theorem}

\begin{proof}
$I\rightarrow pI$ is a 1--1 correspondence between ideals of norm $n$ in $R$ and ideals of norm $p^{2}n$ in $R$. Furthermore if $(n,p)=1$, there are no ideals of norm $pn$. Since $(-1/p) = -1$, the result follows.
\qed
\end{proof}

\begin{lemma}
\label{lemma3.8}
$\theta(\mathit{AMB})$ is in $2Z[[x]]$.
\end{lemma}

\begin{proof}
The ideals $(a+2bqi)$ and $(a-2bqi)$, $b$ odd, are distinct, lie in  $\mathit{AMB}$, and make equal contributions to $\theta(\mathit{AMB})$. And these are all the ideals in $\mathit{AMB}$.
\qed
\end{proof}

Now consider the subspace of $Z/2[[x]]$ spanned by the mod 2 reductions of the following elements of $Z[[x]]$: the $\theta(R)$ and $\frac{1}{2}\theta(\mathit{AMB})$.

\begin{definition}
\label{def3.9}
$DI(q)$ is the image of the above space under the map $p_{3,2}$ of Definition \ref{def2.12}.
\end{definition}

\begin{theorem}
\label{theorem3.10}
The $T_{p}$, $p\equiv 7\mod{12}$, annihilate $DI(q)$.
\end{theorem}

\begin{proof}
This is immediate from Theorem \ref{theorem3.7}, since the $T_{p}$ of Definition \ref{def3.5} reduces to the characteristic 2 $T_{p}$ of the introduction.
\qed
\end{proof}

\begin{lemma}
\label{lemma3.11}
If $R$ is the principal Gauss-class, then the mod $2$ reduction of $\theta(R)$ is $F$.
\end{lemma}

\begin{proof}
The ideals $(a+4bqi)$ and $(a-4bqi)$, $b\ne 0$, are distinct, lie in $R$, and make equal contributions to $\theta(R)$. Also, the remaining ideals lying in $R$ are the $(n)$ with $n$ odd and $>0$. 
\qed
\end{proof}

\begin{theorem}
\label{theorem3.12}
$DI(q)$ is a vector space over $Z/2$ of dimension $q$.
\end{theorem}

\begin{proof}
Let $C$ be a generator of the Gauss group, so that the Gauss classes are the $C^{i}$, $0\le i <2q$. Note that $C^{q}=\mathit{AMB}$. Since $\theta(C^{i})=\theta(C^{2q-i})$, and $p_{3,2}$ annihilates the reduction, $F$, of $\theta(C^{0})$, the dimension of $DI(q)$ is $\le q$. Let $\alpha_{q}$ be the reduction of $\frac{1}{2}\theta(\mathit{AMB})$, and let $\alpha_{i}$, $0\le i < q$, be the reduction of $\theta(C^{i})$. If $\beta_{i}=p_{3,2}(\alpha_{i})$, then $\beta_{0}=0$, while $\beta_{1},\ldots , \beta_{q}$ span $DI(q)$. It remains to show their linear independence. Suppose that $R$ is any Gauss-class. Dirichlet's theorem for prime ideals in $Z[i]$ shows that $R$ contains a prime ideal $P=(a+bi)$ with $a+bi \equiv 1+i \mod{3}$. Then the norm of $P$ is a prime $p \equiv 2\mod{3}$. Since the only ideals of this norm are $P$ and $\bar{P}$, the only Gauss-classes containing ideals of norm $p$ 
are $R$ and $R^{-1}$. Now take $R$ to be $C^{j}$ with $1\le j \le q$. Then the coefficient of $x^{p}$ in this $\theta(C^{j})$ is 1 if $j<q$ and 2 if $j=q$, while the coefficient of $x^{p}$ in each of the other $\theta(C^{k})$, $1\le k\le q$, is 0. So the coefficient of $x^{p}$ in $\beta_{k}$ is 1 if $k=j$, and 0 otherwise.
\qed
\end{proof}

\begin{theorem}
\label{theorem3.13}
The $T_{p}$, $p\equiv 1\mod{6}$, stabilize $DI(q)$.
\end{theorem}

\begin{proof}
For $p\equiv 7\mod{12}$ we use Theorem \ref{theorem3.10}. Suppose $p\equiv 1\mod{12}$, so that there is an ideal $P$ of norm $p$. Then for any Gauss-class, $R$, $T_{p}(\theta(R))=\theta(PR)+\theta(\bar{P}R)$, while $T_{p}(\frac{1}{2}\theta(\mathit{AMB}))=\theta(P\cdot\mathit{AMB})$. Reducing mod 2, applying $p_{3,2}$, and noting that the $T_{p}$ of Definition \ref{def3.5} reduces to the characteristic 2 $T_{p}$ of the introduction gives the result.
\qed
\end{proof}

Now let $p$ be congruent to 13 mod 24. Make $DI(q)$ into a $Z/2[Y]$-module with $Y$ acting by $T_{p}$. We'll determine the structure of this module.  To this end we take a Gauss-class $C$ containing an ideal of norm $p$.  $C$ generates the Gauss group, and we use the construction of Theorem \ref{theorem3.12} with this $C$ to define $\beta_{0}=0$ and the basis $\beta_{1},\ldots,\beta_{q}$ of $DI(q)$.

\begin{definition}
\label{def3.14}
For $n\ge 1$, $U_{n}$ is the characteristic $2$ polynomial such that $U_{n}(t+t^{-1})=t^{n}+t^{-n}$.
\end{definition}

Note that $U_{1}(Y)=Y$, that $U_{2n}=U_{n}^{2}$, and that $U_{n+2}(Y)=YU_{n+1}(Y)+U_{n}(Y)$.

\begin{theorem}
\label{theorem3.15}
Adopt the notation of the paragraph preceding Definition \ref{def3.14}. Then:

For $1\le i \le q$, $\beta_{q-i}=U_{i}(Y)\cdot \beta_{q}$, with the $U_{i}$ as in Definition \ref{def3.14}.
\end{theorem}

\begin{corollary}
\label{corollary3.15}
In the above situation, $DI(q)$ is isomorphic to $Z/2[Y]/(Y^{q})$ as $Z/2[Y]$-module; furthermore $\beta_{q}$ generates the module.
\end{corollary}

\begin{proof}
Since the $\beta_{i}$ span $DI(q)$, Theorem \ref{theorem3.15} shows that the module is cyclic with generator $\beta_{q}$. Also, $U_{q}(Y)=Y^{q}$ and so $(Y^{q})\cdot\beta_{q}=\beta_{q-q}=0$. Since $DI(q)$ is cyclic of $Z/2$-dimension $q$, annihilated by $Y^{q}$, it is isomorphic to $Z/2[Y]/(Y^{q})$.
\qed
\end{proof}

\begin{proof}[of Theorem \ref{theorem3.15}]
$T_{p}(\theta(\mathit{AMB}))=T_{p}(\theta(C^{q}))=2\theta(C^{q-1})$. Dividing by 2, reducing, and applying $p_{3,2}$ we find that $T_{p}(\beta_{q})=\beta_{q-1}$. So $\beta_{q-1}=Y\beta_{q}$. Also, $T_{p}(\theta(C^{q-1}))=\theta(\mathit{AMB})+\theta(C^{q-2})$. Reducing and applying $p_{3,2}$ we find that $\beta_{q-2}=Y\beta_{q-1}=Y^{2}\beta_{q}$. It remains to show that if $0\le j \le q-3$ then $\beta_{j}=Y\beta_{j+1}+\beta_{j+2}$. But this is proved in the same way---
$T_{p}(\theta(C^{j+1}))=\theta(C^{j+2})+\theta(C^{j})$, and again we reduce and apply $p_{3,2}$.
\qed
\end{proof}

We conclude this section by connecting the space $DI(q)$ to the space $W5$ of the last section, and illustrating with an example.

\begin{lemma}
\label{lemma3.17}
The mod $2$ reduction of $\frac{1}{2}\theta(\mathit{AMB})$ is $F^{4q^{2}+1}$.
\end{lemma}

\begin{proof}
The ideals in $\mathit{AMB}$ are the $(a+2bqi)$ with $a$ and $b$ odd. Each choice of $|a|$ and $|b|$ gives two ideals in $\mathit{AMB}$. So $\frac{1}{2}\theta(\mathit{AMB})$ is the product of $\sum_{\, a\ \mathrm{odd},\ a>0}x^{a^{2}}$ and $\sum_{\, b\ \mathrm{odd},\ b>0}x^{4b^{2}q^{2}}$. Reducing mod $2$ we get $F\cdot F^{4q^{2}}$.
\qed
\end{proof}

\begin{theorem}
\label{theorem3.18}
$\beta_{q}=D^{4q^{2}+1}$.
\end{theorem}

\begin{proof}
By Lemma \ref{lemma3.17} this amounts to showing that $p_{3,2}(F^{4q^{2}+1})=D^{4q^{2}+1}$. Now $F^{4q^{2}+1}=(D+H)^{4q^{2}+1}=D^{4q^{2}+1}+H\cdot D^{4q^{2}}+H^{4q^{2}}\cdot D + H^{4q^{2}+1}$, and $p_{3,2}$ annihilates the last 3 terms in the sum.
\qed
\end{proof}

\begin{theorem}
\label{theorem3.19}
Let $W5(q)$ be the subspace of $W5$ spanned by the $D^{n}$ with $n\equiv 5\mod{6}$ and $< 12q^{2}$. Then $DI(q)\subset W5(q)$.
\end{theorem}

\begin{proof}
Take $C$ to be the class of $(3+2i)$ in the construction of Theorem \ref{theorem3.12}. Then $\beta_{q}=D^{4q^{2}+1}$ lies in
$W5(q)$.  Furthermore $T_{13}$ stabilizes $W5(q)$, by Theorem \ref{theorem2.17}. When we make $DI(q)$ into a $Z/2[Y]$-module with $Y$ acting by $T_{13}$, then the element $\beta_{q}$ of $W5(q)$ is a generator, and the theorem follows.
\qed
\end{proof}

Here's a summary of the results of this section that are relevant to what will follow: There is for each power $q$ of 2 a $q$-dimensional subspace of $W5(q)$, stable under $X=T_{7}$ and $Y=T_{13}$, and annihilated by $X$. If we view this subspace as a $Z/2[Y]$-module it is cyclic.

\begin{example*}{Example}
Let $q = 8$ and $C=(3+2i)$. We've shown that $\beta_{1},\ldots , \beta_{8}$ lie in $W5(q)$. Here they are explicitly.

\parbox{\textwidth}
{\centering
\begin{eqnarray*}
\beta_{1} &=& D^{245}+D^{221}+D^{197}+D^{125}+D^{101}\\
\beta_{2} &=& D^{209}+D^{113}+D^{65}+D^{41}+D^{17}\\
\beta_{3} &=& D^{245}+D^{221}+D^{125}+D^{77}+D^{29}+D^{5}\\
\beta_{4} &=& D^{65}\\
\beta_{5} &=& D^{245}+D^{221}+D^{125}+D^{77}+D^{53}\\
\beta_{6} &=& D^{209}+D^{113}+D^{65}+D^{41}\\
\beta_{7} &=& D^{245}+D^{221}+D^{197}+D^{125}+D^{101}+D^{53}+D^{29}\\
\beta_{8} &=& D^{257}\\
\end{eqnarray*}
}
\end{example*}

\section{The action of $\bm{T_{7}}$ and $\bm{T_{13}}$ on $\bm{W5}$}
\label{section4}

Since 7 and 13 are each $\equiv 1\mod{6}$, $T_{7}$ and $T_{13}$ stabilize $W5$.  Make $W5$ into a $Z/2[X,Y]$-module with $X$ and $Y$ acting by $T_{7}$ and $T_{13}$. Since neither 7 nor 13 is 1 mod 24, Theorem \ref{theorem2.17} shows that $T_{7}(D^{6m+5})$ and $T_{13}(D^{6m+5})$ are sums of $D^{6r+5}$ with $r<m$, so $(X,Y)^{m+1}$ annihilates $D^{6m+5}$, and $W5$ has the structure of $Z/2[[X,Y]]$-module.  We'll use Theorem \ref{theorem2.24} and the results of section \ref{section3} to construct a $Z/2$-basis $m_{a,b}$ of $W5$ ``adapted to $T_{7}$ and $T_{13}$''.

Throughout, $g$ is the function $N\rightarrow N$ of Definition \ref{def2.21} with $g(2n)=4g(n)$ and $g(2n+1)=g(2n)+1$. Also $[a,b,G]$ is $D^{n}$ where $n=5 + 6g(a)+12g(b)$.

\begin{lemma}
\label{lemma4.1}
$g(r+s)\ge g(r)+g(s)$.
\end{lemma}

\begin{proof}
We argue by induction on $r+s$, noting that $g(0)=0$. If $r$ is odd we replace $r$ by $r-1$. If $s$ is odd we replace $s$ by $s-1$. If $r$ and $s$ are even and not both 0, we replace them by $r/2$ and $s/2$.
\qed
\end{proof}

\begin{lemma}
\label{lemma4.2}
Let $D^{m}=[c,d,G]$ and $D^{n}=[0,b,G]$. If $D^{m}$ is earlier than $D^{n}$ (see Definition \ref{def2.21}, and the ordering of the $D^{k}$ described in Definition \ref{def2.23}), then $m<n$.
\end{lemma}

\begin{proof}
$n=5+12g(b)\ge 5+12g(c+d)$, while $m=5+6g(c)+12g(d)$. So by Lemma \ref{lemma4.1}, $m\le n$. But as $D^{m}$ is earlier than $D^{n}$, $m\ne n$.
\qed
\end{proof}

\begin{lemma}
\label{lemma4.3}
Suppose $f\ne 0$ is in $W5(q)$ and $Xf=0$. Write $f$ as $[a,b,G]\, +$ a sum of earlier monomials in $D$. Then $a=0$ and $b<q$.
\end{lemma}

\begin{proof}
If $a>0$, then by Theorem \ref{theorem2.24}, $0=Xf = [a-1,b,G] +$ a sum of earlier monomials, a contradiction. So $f=[0,b,G]+$ a sum of earlier monomials. Lemma \ref{lemma4.2} then shows that, as a polynomial in $D$, $f$ has degree $5+12g(b)$. If $b\ge q$ then $f$ has degree $\ge 5+12g(q)=5+12q^{2}$ and so is not in $W5(q)$.
\qed
\end{proof}

\begin{theorem}
\label{theorem4.4}
The kernel of $X : W5(q)\rightarrow W5(q)$ is $DI(q)$.
\end{theorem}

\begin{proof}
Lemma \ref{lemma4.3} shows that the dimension of the kernel is at most $q$. But we've seen that $DI(q)$ is a $q$-dimensional subspace of $W5(q)$ contained in the kernel.
\qed
\end{proof}

\begin{corollary}
\label{corollary4.5}
$DI(1)\subset DI(2)\subset DI(4)\subset \ldots$, and the kernel of $X : W5\rightarrow W5$ is the union, $DI$, of the $DI(q)$.
\end{corollary}

\begin{theorem}
\label{theorem4.6}
The only elements of $W5$ annihilated by $(X,Y)$ are $0$ and $D^{5}$.
\end{theorem}

\begin{proof}
If $(X,Y)f = 0$, $f$ is in some $DI(q)$. By
the results of the last section, $DI(q)$, viewed as $Z/2[Y]$ module, is isomorphic to $Z/2[Y]/(Y^{q})$. So $Y : DI(q)\rightarrow DI(q)$ has 1-dimensional kernel.
\qed
\end{proof}

\begin{definition}
\label{def4.7}
$S_{m}$ is the subspace of $W5$ of dimension $m(m+1)/2$ spanned by the monomials $[a,b,G]$, $a+b<m$.
\end{definition}

Note that $S_{0}=(0)$ while $S_{1}$ is spanned by $[0,0,G]=D^{5}$. So $X\cdot S_{1}=Y\cdot S_{1}=S_{0}$.

\begin{lemma}
\label{lemma4.8}
$X : W5\rightarrow W5$ is onto.  In fact, $X$ maps $S_{m+1}$ onto $S_{m}$.
\end{lemma}

\begin{proof}
By Theorem \ref{theorem2.24}, $X\cdot S_{m+1}\subset S_{m}$, so it suffices to show that the dimension of the kernel of $X : S_{m+1}\rightarrow S_{m}$ is $\le m+1$. Suppose $f\ne 0$ is in this kernel.  The proof of Lemma \ref{lemma4.3} shows that $f=[0,b,G]+$ a sum of earlier monomials in $D$, and that the degree of $f$ in $D$ is $5+12g(b)$. But Lemma \ref{lemma4.2} shows that every element of $S_{m+1}$ has degree $\le 5+12g(m)$. So $0\le b\le m$, giving the result.
\qed
\end{proof}

\begin{theorem}
\label{theorem4.9}
$Y\cdot S_{m+1} \subset S_{m}$.
\end{theorem}

\begin{proof}
We argue by induction on $m$, $m=0$ being clear. Suppose $f$ is in $S_{m+1}$, $m>0$. Then $Xf$ is in $S_{m}$, so by induction, $X(Yf)=Y(Xf)$ is in $S_{m-1}$. By
Lemma \ref{lemma4.8} there is an $h$ in $S_{m}$ such that $X(h+Yf)=0$, and we only need show that $h+Yf$ is in $S_{m}$. If $h+Yf \ne 0$, write it as $[a,b,G]+$ a sum of earlier monomials. Then $a=0$, and we need to show that $b<m$. Now since $f$ is in $S_{m+1}$, its degree in $D$ is $\le 5+12g(m)$. As $T_{13}$ is degree-decreasing, $Yf$ has degree $<5+12g(m)$. But $Yf = h+([0,b,G]+ \mbox{a sum of earlier monomials})$. If $b\ge m$, the right hand side of this equality has degree $5+12g(b)\ge 5+12g(m)$, a contradiction.
\qed
\end{proof}

\begin{lemma}
\label{lemma4.10}
For each $m$ there is an element of $DI$ of the form $[0,m,G]+$ a sum of earlier monomials.
\end{lemma}

\begin{proof}
Fix $q>m$. Then every $f\ne 0$ in $DI(q)$ can be written as $[0,b,G]+$ a sum of earlier monomials, for some $b$ with $0\le b<q$. Since there are only $q$ possible choices of $b$ and $DI(q)$ has dimension $q$, the result follows.
\qed
\end{proof}

\begin{lemma}
\label{lemma4.11}
$DI\cap S_{m}$ has dimension $m$. Furthermore, $Y$ maps $DI\cap S_{m+1}$ onto $DI\cap S_{m}$.
\end{lemma}

\begin{proof}
By Lemma \ref{lemma4.10}, $DI\cap S_{m+1}\ne DI\cap S_{m}$. Now $Y$ maps $DI\cap S_{m+1}$ into $DI\cap S_{m}$, and by Theorem \ref{theorem4.6} the kernel of this map is contained in $\{0, D^{5}\}$. So the map is onto, and the dimensions of $DI\cap S_{m+1}$
and $DI\cap S_{m}$ differ by~1.
\qed
\end{proof}

\begin{theorem}
\label{theorem4.12}
Let $f$ and $h$ be elements of $S_{m}$ with $Yf=Xh$. Then there is an $e$ in $S_{m+1}$ with $Xe =f$ and $Ye =h$.
\end{theorem}

\begin{proof}
There is an $e_{1}$ in $S_{m+1}$ with $Xe_{1} =f$ by Lemma \ref{lemma4.8}.  Replacing $f$ and $h$ by $f+Xe_{1}$ and $h+Ye_{1}$ we may assume that $f=0$. Then $Xh=Yf=0$, $h$ is in $DI\cap S_{m}$ and we apply Lemma \ref{lemma4.11}.
\qed
\end{proof}

\begin{corollary}
\label{corollary4.13}
There are $m_{a,b}$ in $S_{a+b+1}$ such that:
\begin{enumerate}
\item[(1)] $m_{0,0}=D^{5}$.
\item[(2)] $Xm_{a,b}=m_{a-1,b}$ or $0$ according as $a>0$ or $a=0$.
\item[(3)] $Ym_{a,b}=m_{a,b-1}$ or $0$ according as $b>0$ or $b=0$.
\end{enumerate}
\end{corollary}

\begin{proof}
We construct the $m_{a,b}$ inductively, by induction on $a+b$, taking $m_{0,0}=D^{5}$. Note that $Xm_{0,0}=Ym_{0,0}=0$. Suppose the $m_{a,b}$ are defined for $a+b<r$, and that $a+b=r$. If neither $a$ nor $b$ is $0$ let $m_{a,b}$ be any $e$ in $S_{r+1}$ with $Xe = m_{a-1,b}$, $Ye = m_{a,b-1}$; such $e$ exists by the theorem.  Finally let $m_{r,0}$ be any $e$ in $S_{r+1}$ with $Xe = m_{r-1,0}$ and $Ye=0$, and let $m_{0,r}$ be any $e$ in $S_{r+1}$ with $Xe =0$, $Ye =m_{0,r-1}$.
\qed
\end{proof}

\begin{theorem}
\label{theorem4.14}
If $m_{a,b}$ are as in Corollary \ref{corollary4.13},
then for each $r$, the $m_{a,b}$ with $a+b\le r$ are linearly independent (and since there are $(r+1)(r+2)/2$ of them, they form a basis of $S_{r+1}$).
\end{theorem}

\begin{proof}
We argue by induction on $r$, $r=0$ being trivial. Suppose on the contrary that some non-empty sum of distinct $m_{a,b}$ with $a+b=r$ lies in $S_{r}$.  Then the sum of the corresponding $Xm_{a,b}$ is in $S_{r-1}$. By the induction assumption there is only one $m_{a,b}$ in the sum, and this $m_{a,b}$ is $m_{0,r}$. So $m_{0,r}$ is in $S_{r}$; applying $Y$ and using the induction hypothesis gives a contradiction.
\qed
\end{proof}

Theorem \ref{theorem4.14} tells us that the $m_{a,b}$ are a basis of $W5$; we say that they constitute ``a basis adapted to $T_{7}$ and $T_{13}$''.

\begin{theorem}
\label{theorem4.15}
$Z/2[[X,Y]]$ acts faithfully on $W5$.
\end{theorem}

\begin{proof}
Suppose $u\ne0$ is in $Z/2[[X,Y]]$. Take $k$ so that $u$ is in $(X,Y)^{k}$, but not in $(X,Y)^{k+1}$.  Then a monomial $X^{a}Y^{b}$ with $a+b=k$ appears in $u$, and $u\cdot m_{a,b}=m_{0,0}\ne 0$.
\qed
\end{proof}

\begin{theorem}
\label{theorem4.16}
If $T : W5\rightarrow W5$ is a $Z/2[[X,Y]]$-linear map, then $T$ is multiplication by some $u$ in $Z/2[[X,Y]]$.
\end{theorem}

\begin{proof}
Since $X^{k+1}$ and $Y^{k+1}$ annihilate $m_{k,k}$ they annihilate $T(m_{k,k})$. Writing $T(m_{k,k})$ as a sum of distinct $m_{a,b}$ we see that each $a$ and each $b$ are $\le k$.  It follows from this that $T(m_{k,k})=u_{k}\cdot m_{k,k}$ for some $u_{k}$ in $Z/2[X,Y]$. Then $T(m_{a,b})=u_{k}\cdot m_{a,b}$ whenever $a$ and $b$ are $\le k$, and in particular $T(f)=u_{k}\cdot f$ for all $f$ in $S_{k}$. The $u_{k}$ form a Cauchy sequence in $Z/2[[X,Y]]$, and the limit, $u$, of this sequence has the desired property. 
\qed
\end{proof}

\begin{theorem}
\label{theorem4.17}
If $p\equiv 1\mod{6}$, $T_{p} : W5\rightarrow W5$ is multiplication by some $u$ in the ideal $(X,Y)$. In other words, $T_{p}$, in its action on $W5$, is a power series with 0 constant term in $T_{7}$ and $T_{13}$.
\end{theorem}

\begin{proof}
$T_{p}$ commutes with $X=T_{7}$ and $Y=T_{13}$. So it is $Z/2[[X,Y]]$-linear, and is multiplication by some $u$. Let $c$ be the constant term of $u$. Since $X$ and $Y$ annihilate $D^{5}$, $T_{p}(D^{5})=c\cdot D^{5}$. Applying $T_{5}$ we find that $T_{p}(D)=c\cdot D$. Since $T_{p}(D)=0$, $c=0$.
\qed
\end{proof}

Before going on to Theorem \ref{theorem4.18} we make a digression into level 1 theory. Replace $\Gamma_{0}(3)$ by the full modular group in Definition \ref{def2.2}. Using the fact that the mod~2 reductions of the expansions at infinity of $E_{4}$ and $\Delta$ are $1$ and $F$, together wtih dimension formulas, one finds that $M_{12}$ has the basis $\{1,F\}$, and more generally that the $F^{i}$, $0 \le i \le k$ are a basis of $M_{12k}$. Consider the space spanned by the $F^{i}$, $i>0$ and odd; an $f$ in $Z/2[[x]]$ lies in this space precisely when $f$ is odd and in some $M_{k}$. So the space is the level 1 analog of the $M(\mathit{odd})$ of Definition \ref{def2.6} (and at the same time the level 1 analog of $W=W1 \oplus W5$). The modular forms interpretation of our space shows that it is stabilized by the $T_{p}$, $p$ an odd prime, and that in fact $T_{p}(F^{n})$ is a sum of $F^{k}$ with each $k\le n$ and $\equiv pn \mod{8}$. (One consequence of this is that $T_{3}$ takes $F$, $F^{3}$, $F^{5}$, $F^{7}$ to $0$, $F$, $0$, $F^{5}$.) We may view our space as a $Z/2[[X,Y]]$-module with $X$ and $Y$ acting by $T_{3}$ and $T_{5}$. The level~1 Hecke-algebra structure theorem, \cite{4}, says that the action of $Z/2[[X,Y]]$ is faithful, and that each $T_{p}$ acts by multiplication by an element of the maximal ideal $(X,Y)$.

This digression sketches a new and simpler proof of the above result of Nicolas and Serre. Fix a power $q$ of 2. The Gauss group of section \ref{section3} is cyclic of order $2q$. We have attached to each Gauss-class $R$ a $\theta(R)$ in $Z[[x]]$. Let $\alpha(R)$ be the mod~2 reduction of $\theta(R)$. Consider the space spanned by the $\alpha(R)$; modifying our notation we call this space $DI(q)$. Arguing as in Theorems \ref{theorem3.10} and \ref{theorem3.13} we find that the $T_{p}$ of Definition \ref{def3.5} with $p\equiv 3\mod{4}$ annihilate $DI(q)$ while those with $p\equiv 1\mod{4}$ stabilize $DI(q)$. Fix the generator $C=(1+2i)$ of the Gauss group (see Theorem \ref{theorem3.2}), and for $0\le i<q$ let $\alpha_{i}=\alpha(C^{i})$. Since $\alpha(C^{2q-i}) = \alpha(C^{i})$ and $\alpha(C^{q}) = \alpha(\mathit{AMB})=0$, the $\alpha_{i}$ span  $DI(q)$. By Lemma \ref{lemma3.11}, $\alpha_{0}=F$. Also, $T_{5}(\theta(C^{q}))=2\theta(C^{q-1})$ by Theorem \ref{theorem3.6}. Dividing by 2, reducing mod~2 and using Lemma \ref{lemma3.17} we find:

\begin{digressionresult}
$\alpha_{q-1}=T_{5}(F^{4q^{2}+1})$.
\end{digressionresult}

Now make $DI(q)$ into a $Z/2[X,Y]$-module with $X$ and $Y$ acting by $T_{3}$ and $T_{5}$. Then $Y\cdot \alpha_{0} = Y\cdot F = 0$. Also, Theorem \ref{theorem3.6} shows that for $0<i<q-1$, $Y\cdot \alpha_{i} = \alpha_{i-1}+\alpha_{i+1}$, while $Y\cdot \alpha_{q-1} = \alpha_{q-2}$. Arguing as in the proof of Theorem \ref{theorem3.15} we find that $\alpha_{q-1}$ generates $DI(q)$ as $Z/2[Y]$-module, and that $Y^{q-1}\cdot \alpha_{q-1} =\alpha_{0}=F$. So $Y^{q}\cdot\alpha_{q-1}=0$, and we've shown:

\begin{digressionresult}
$DI(q)$ has dimension $q$. It is annihilated by $X$, and as $Z/2[Y]$-module is isomorphic to  $Z/2[Y]/Y^{q}$, with a generator being $T_{5}(F^{4q^{2}+1})$.
\end{digressionresult}

Since $F$ is transcendental over $Z/2$, we may identify the space spanned by the $F^{k}$, $k>0$ and odd, with the space of Definition \ref{def2.21}(2) in the obvious way; accordingly we call the space $V$. (The $[a,b]$ of Definition \ref{def2.21} is then $F^{1+2g(a)+4g(b)}$, and we have a total ordering on the $F^{k}$, $k>0$ and odd, provided by Definition \ref{def2.21}(3).)

Now let $V(q)\subset V$ be the supspace spanned by the $F^{i}$, $i$ odd and $<4q^{2}$. Recall that $T_{p}(F^{n})$ is a sum of $F^{k}$ with $k\le n$ and $\equiv pn \mod{8}$; it follows that the $T_{p}$ stabilize $V(q)$, and also that $T_{5}(F^{4q^{2}+1})$ lies in $V(q)$. (2) above then shows:

\begin{digressionresult}
$DI(q)\subset V(q)$.
\end{digressionresult}

$V$ is a $Z/2[X,Y]$-module with $X$ and $Y$ acting by $T_{3}$ and $T_{5}$, and $V(q)$ is a  $Z/2[X,Y]$-submodule. We now study the action of $X$ on $V(q)$. Observe that if $u$ is in $Z/2[[x]]$ then $T_{3}(F^{8}u) = F^{8}T_{3}(u) + F^{2}T_{3}(F^{2}u)$. (The proof is just like that of Lemma \ref{lemma2.19}, using the level 3 modular equation, $F^{4} + G^{4} + FG = 0$, for  $F$.) Taking $u=F^{n}$, $n$ odd, and letting $A_{n}$ be the element of $X\cdot F^{n}$ of $V$ we find that $A_{n+8}=F^{8}A_{n}+F^{2}A_{n+2}$. And as we've seen, $A_{1}$, $A_{3}$, $A_{5}$ and $A_{7}$ are $0$, $F$, $0$, and $F^{5}$. Corollary 3.8 of \cite{2} then gives the following result (which is Proposition 4.3 of \cite{3}):

\begin{digressionresult}
If $F^{n}=[a,b]$ then $A_{n}$ is a sum of $[c,d]$ with $c+d<a+b$. Furthermore if $a>0$, $A_{n}=[a-1,b]+$ a sum of earlier monomials in~$F$.
\end{digressionresult}

Suppose now that $f \ne 0$ is in $V(q)$ and that $X\cdot f = 0$. Write $f$ as $[a,b]+$ a sum of earlier monomials. If $a>0$, (4) shows that $[a-1,b]+$ a sum of earlier monomials is $0$, and we get a contradiction. So $f=[0,b]+$ a sum of earlier monomials in $F$. Then as a polynomial in $F$, the degree of $f$ is $1+4g(b)$. If $b\ge q$, the degree of $f$ is $\ge 4q^{2}+1$, contradicting the fact that $f$ is in $V(q)$. As in the proof of Theorem \ref{theorem4.4} we conclude that the kernel of $X: V(q)\rightarrow V(q)$ has dimension at most $q$ and so:

\begin{digressionresult}
The kernel of $X: V(q)\rightarrow V(q)$ is $DI(q)$.
\end{digressionresult}

Our machinery is now in place. We argue as in Corollary \ref{corollary4.5} through Theorem \ref{theorem4.14} of this section, defining $S_{m}$ to be the subspace of $V$ spanned by the $[a,b]$ with $a+b< m$, and we prove:

\begin{digressionresult}
There is a basis $m_{a,b}$ of $V$, $(a,b)$ in $N\times N$, with $m_{0,0}=F$ and:\\
\vspace{-5ex}
\begin{list}{}{\leftmargin=1in}
\item[(1)] $X\cdot m_{a,b} = m_{a-1,b}$ or $0$ according as $a>0$ or $a=0$.
\item[(2)] $Y\cdot m_{a,b} = m_{a,b-1}$ or $0$ according as $b>0$ or $b=0$.
\end{list}
\end{digressionresult}

Then, arguing as in Theorems \ref{theorem4.15}, \ref{theorem4.16} and \ref{theorem4.17}, we find that $Z/2[[X,Y]]$ acts faithfully on $V$ and that each $T_{p}:V\rightarrow V$ is multiplication by some element of the maximal ideal $(X,Y)$. Note that our argument has avoided any use of the highly technical Proposition 4.4 of \cite{3}. We return to level $\Gamma_{0}(3)$.


\begin{theorem}
\label{theorem4.18}
There is a $\lambda$ in $(X,Y)$ such that $T_{5}^{2}$, in its action on $W5$, is multiplication by $\lambda^{2}$.
\end{theorem}

\begin{proof}
As in the proof of Theorem \ref{theorem4.17} we see that $T_{5}^{2}$ is multiplication by some $u$ in $(X,Y)$. Write $u$ as $a+bX+cY + dXY$, where $a$, $b$, $c$ and $d$ are power series in $X^{2}$ and $Y^{2}$, and let $a=\lambda^{2}$. Suppose
that $e$ is in $W5$; we'll show that $T_{5}^{2}(e)=\lambda^{2}e$. We may assume that $e$ is a sum of various $D^{k}$, where all the $k$ appearing are congruent to one another mod $24$. To illustrate, suppose $e$ is a sum of $D^{k}$, $k\equiv 5\mod{24}$. Then $T_{5}^{2}(e)$ and $ae$ are each sums of $D^{k}$, $k\equiv 5\mod{24}$, $(bX)e$ is a sum of $D^{k}$, $k\equiv 11\mod{24}$, $(cY)e$ is a sum of $D^{k}$, $k\equiv 17\mod{24}$, and $(dXY)e$ is a sum of $D^{k}$, $k\equiv 23\mod{24}$.  Since $T_{5}^{2}(e)=ue$ is the sum of $\lambda^{2}e$, $(bX)e$, $(cY)e$ and $(dXY)e$, the result follows.
\qed
\end{proof}

We'll conclude this section with a result that allows us to pass from $W5$ to~$W1$.

\begin{lemma}
\label{lemma4.19}
For $u$ in $Z/2[[x]]$, $T_{5}(G^{16}u)=G^{16}T_{5}(u)+ G^{2}T_{5}(uG^{6})+G^{6}T_{5}(uG^{2})$.
\end{lemma}

\begin{proof}
Let $U$ be the 2-variable polynomial $(A+B)^{6}+AB$ over $Z/2$. Then $U(F(x^{5}), F(x))=0$; this is the modular equation of level 5 for $F$. So \linebreak$U(G(x^{5}), G(x))=0$, and if we set $V=(A+B)^{2}\cdot U= A^{8}+B^{8}+A^{3}B+AB^{3}$, then $V(G(x^{5}),G(x))=0$.  Now continue as in the proof of Lemma \ref{lemma2.19}.
\qed
\end{proof}

\begin{theorem}
\label{theorem4.20}
$T_{5}(D^{6m+5})=D^{6m+1}+$ a sum of $D^{6r+1}$, $r<m$. In particular, the map $T_{5} : W5\rightarrow W1$ (see Definition \ref{def1.1} and Theorem \ref{theorem2.16}) is 1--1 and onto.
\end{theorem}

\begin{proof}
Applying Lemma \ref{lemma4.19} to $u=D^{6n+5}$ we find that $T_{5}(D^{6n+53})=D^{48}T_{5}(D^{6n+5})+D^{6}T_{5}(D^{6n+23})+D^{18}T_{5}(D^{6n+11})$. So if the result holds for $m=n$, $n+1$ and $n+3$, it also holds for $m=n+8$, and it's enough to prove the result for $m<8$. This may be done by direct calculation; an illustration is given in the proof of Lemma \ref{lemma2.18}.
\qed
\end{proof}

Theorem \ref{theorem4.18} now gives:

\begin{corollary}
\label{corollary4.21}
Let $\lambda$ be as in Theorem \ref{theorem4.18}. Then, in its action on $W=W5+W1$, $T_{5}^{2}$ is multiplication by $\lambda^{2}$.
\end{corollary}

\section{The algebra $\bm\newo$, acting on $\bm W$}
\label{section5}

Take $\lambda$ in $(X,Y)$ as in Corollary \ref{corollary4.21}. By Theorem \ref{theorem2.17}, a three-variable power series ring over $Z/2$ acts on $W$, with the variables acting by $T_{7}$, $T_{13}$ and $\lambda(T_{7},T_{13})+T_{5}$. Now Corollary \ref{corollary4.21} shows that the square of the third variable annihilates $W$. So if we let $\newo$ be $Z/2[[X,Y]]$ with an element $\varepsilon$ of square 0 adjoined, then there is an action of the local ring $(\newo, m)$ on $W$ with $X$, $Y$ and $\varepsilon$ acting by $T_{7}$, $T_{13}$ and $\lambda(T_{7},T_{13})+T_{5}$.

\begin{lemma}
\label{lemma5.1}
The only element of $W5$ annihilated by $\varepsilon$ is $0$.
\end{lemma}

\begin{proof}
Suppose $\varepsilon h_{5}=0$, $h_{5}$ in $W5$. Then $T_{5}(h_{5})=\lambda(X,Y)\cdot h_{5}$. Since the first of these is in $W1$ and the second in $W5$, $T_{5}(h_{5})=0$, and we apply Theorem \ref{theorem4.20}.
\qed
\end{proof}

\begin{theorem}
\label{theorem5.2}
A $Z/2$-basis of $W$ is given by the $m_{a,b}$ and the $\varepsilon\cdot m_{a,b}$, with the $m_{a,b}$ as in Theorem \ref{theorem4.14}.
\end{theorem}

\begin{proof}
If $h_{1}$ is in $W1$, then by Theorem \ref{theorem4.20}, $h_{1}=T_{5}(h_{5})=\lambda(X,Y)\cdot h_{5}+\varepsilon\cdot h_{5}$ for some $h_{5}$ in $W5$. So $h_{1}$ is the sum of an element of $W5$ and an element of $\varepsilon W5$ and the $m_{a,b}$ and $\varepsilon m_{a,b}$ span $W$.  Note also that the sum of $W5$ and $\varepsilon W5$ is direct. For if $h_{5}$ is in both spaces, $\varepsilon h_{5}=0$, and Lemma \ref{lemma5.1} applies. So if there is a non-trivial linear relation between the $m_{a,b}$ and the $\varepsilon m_{a,b}$, there is such a relation between the
$\varepsilon m_{a,b}$, and this is precluded by Lemma \ref{lemma5.1}.
\qed
\end{proof}

\begin{theorem}
\label{theorem5.3}
$\newo$ acts faithfully on $W$.
\end{theorem}

\begin{proof}
Suppose $r+t\varepsilon$, with $r$ and $t$ in $Z/2[[X,Y]]$, annihilates $W$. Then it annihilates $\varepsilon W5$, and so $\varepsilon(rW5)=(0)$. By Lemma \ref{lemma5.1}, $rW5=(0)$, and so $r=0$. Then $t\varepsilon$ annihilates $W$. So $\varepsilon(tW5)=(0)$, $tW5=0$ and $t=0$.
\qed
\end{proof}

\begin{theorem}
\label{theorem5.4}
Each $T_{p}$, $p>3$, acts on $W$ by multiplication by some $r+t\varepsilon$ with $r$ and $t$ in $Z/2[[X,Y]]$.
\end{theorem}

\begin{proof}
$T_{p}$ commutes with $X$, $Y$ and $\varepsilon$. Since $X^{k+1}$ and $Y^{k+1}$ annihilate $m_{k,k}$, they annihilate $T_{p}(m_{k,k})$. Theorem \ref{theorem5.3} now shows that $T_{p}(m_{k,k})$ is a sum of various $m_{i,j}$ and various $\varepsilon\cdot m_{i,j}$, with each $i$ and each $j$ $\le k$. The proof now follows that of Theorem \ref{theorem4.16}, but now we have two Cauchy sequences in $Z/2[[X,Y]]$, one converging to the desired $r$ and the other to the desired $t$.
\qed
\end{proof}

\begin{theorem}
\label{theorem5.5} \hspace{1em}\\
\vspace{-5ex}
\begin{enumerate}
\item[(1)] If $p\equiv 1\mod{6}$, $T_{p} : W\rightarrow W$ is multiplication by some $t$ in the maximal ideal $(X,Y)$ of $Z/2[[X,Y]]$.
\item[(2)] If $p\equiv 5\mod{6}$, $T_{p} : W\rightarrow W$ is the composition of $T_{5}$ with multiplication by some $t$ in  $Z/2[[X,Y]]$.
\end{enumerate}
\end{theorem}

\begin{proof}\hspace{1em}\\
\vspace{-5ex}
\begin{enumerate}
\item[(1)] We saw in Theorem \ref{theorem4.16} that there is a $t$ in $(X,Y)$ such that $T_{p}(h)=t\cdot h$ for all $h$ in $W5$. This identity still holds for $h$ in $W5+T_{5}(W5)=W$.
\item[(2)] $T_{5}$ is multiplication by $\lambda + \varepsilon$. By Theorem \ref{theorem5.4}, $T_{p}$ is multiplication by some $r+t\varepsilon$ with $r$ and $t$ in $Z/2[[X,Y]]$. Then $T_{p}+tT_{5}$ is multiplication by $r+\lambda t$. Since $T_{p}+t\cdot T_{5}$ and multiplication by $r+\lambda t$ map $W5$ into $W1$ and $W5$ respectively, $T_{p}+tT_{5}=0$.
\qed\end{enumerate}

\end{proof}

The elements of the maximal ideal of $\newo$ act locally nilpotently on $W$, and it follows that each $T_{p}$, $p>3$, acts locally nilpotently on $W$. Theorems \ref{theorem5.3} and \ref{theorem5.5} tell us that when we complete the Hecke algebra generated by the $T_{p}$ acting on $W$, with respect to the maximal ideal generated by the $T_{p}$, then the completed Hecke algebra we get is just the non-reduced local ring $\newo$.



\end{document}